\documentclass{aptpubarxiv}
 \usepackage{amsmath,natbib,amsfonts,amsbsy,url,enumerate,bm}
 \RequirePackage{hypernat}
 
\usepackage{color}
\usepackage[normalem]{ulem}

\newcommand{\eps}{\varepsilon}

\newcommand{\norm}[1]{\|#1\|}

\newcommand{\law}{\mathcal{L}}

\newcommand{\EV}{\mathrm{E}}
\renewcommand{\Pr}{\mathrm{P}}

\newcommand{\dm}{\mathrm{d}m}
\newcommand{\ds}{\mathrm{d}s}

\newcommand{\EE}{\mathbb{E}}
\newcommand{\RR}{\mathbb{R}}

\newcommand{\WW}{\mathbb{W}}

\newcommand{\ZZ}{\mathbb{Z}}
\newcommand{\NN}{\mathbb{N}}

\newcommand{\M}{\mathcal{M}}
\newcommand{\BFTC}{\mbox{\sc bftc}}

\newcommand{\1}{\boldsymbol{1}}
\usepackage{etoolbox}
\reversemarginpar

\numberwithin{equation}{section}
\allowdisplaybreaks
\authornames{Jan\ss en, A. and Segers, J.}
\shorttitle{Markov tail chains}
\begin{document}

\title{Markov tail chains}

\authorone[University of Hamburg]{A. Janssen}
\addressone{University of Hamburg, Department of Mathematics, Bundesstr.\ 55, 20146 Hamburg, Germany}
\emailone{anja.janssen@math.uni-hamburg.de}
\authortwo[Universit\'e catholique de Louvain]{J. Segers}
\addresstwo{Universit\'e catholique de Louvain, Institut de statistique, Voie du Roman Pays 20, B-1348 Louvain-la-Neuve, Belgium}
\emailtwo{johan.segers@uclouvain.be}

\begin{abstract}
The extremes of a univariate Markov chain with regulary varying stationary marginal distribution and asymptotically linear behavior are known to exhibit a multiplicative random walk structure called the tail chain. In this paper, we extend this fact to Markov chains with multivariate regularly varying marginal distribution in $\mathbb{R}^d$. We analyze both the forward and the backward tail process and show that they mutually determine each other through a kind of adjoint relation. In a broader setting, it will be seen that even for non-Markovian underlying processes a Markovian forward tail chain always implies that the backward tail chain is Markovian as well. We analyze the resulting class of limiting processes in detail. Applications of the theory yield the asymptotic distribution of both the past and the future of univariate and multivariate stochastic difference equations conditioned on an extreme event.
\end{abstract}



\section{Introduction}
\label{S:intro}

Consider a discrete-time, $\mathbb{R}^d$-valued random process $\{ X_t : t = 0, 1, 2, \ldots \}$ defined by the recursive equation
\begin{equation}
\label{E:MC:1}
	X_t = \Phi ( X_{t-1}, \eps_t ), \qquad t = 1, 2, \ldots,
\end{equation}
where
\vspace{0.2cm}
\begin{equation}
\label{E:MC:2}
\mbox{\begin{minipage}[h]{0.90\textwidth}
\begin{itemize}
\item[(i)] $\eps_1, \eps_2, \ldots$ are independent and identically distributed random elements of a measurable space $(\mathbb{E}, \mathcal{E})$ and independent of $X_0$;
\item[(ii)] $\Phi$ is a measurable function from $\RR^d \times \EE$ to $\RR^d$.
\end{itemize}
\end{minipage}}
\end{equation}

If the process $\{ X_t \}$ happens to be stationary, it will be assumed to be defined for all integer $t$. The distribution of $X_0$ is assumed to be multivariate regularly varying.

The aim of the paper is to analyze the special structure of weak limits of the finite-dimensional distributions of the process conditionally on $\|X_0\|$ being large, where $\|\cdot\|$ denotes the Euclidean norm. More precisely, we will investigate the weak limits, called the forward tail chain, of vectors of the form $(X_0, \ldots, X_t)$ given that $\|X_0\|$ exceeds a high threshold. If in addition the process is stationary we will extend this to find the so-called back-and-forth tail chain, which corresponds to the weak limits of vectors of the form $(X_{-s}, \ldots, X_t)$ given that $\|X_0\|$ is large. A close relation of these processes to multivariate regular variation of the whole process has been analyzed in \citet{BS09}. In this article, we are interested in the special form of the processes, in particular the Markovian structure of both the forward and the backward process and how they necessarily determine each other.

The process $\{ X_t \}$ is obviously a discrete-time homogeneous Markov chain. On the other hand, every homogeneous discrete-time Markov chain $\{ X_t \}$ on a complete separable metric space can be represented as in \eqref{E:MC:1}--\eqref{E:MC:2} \citep{Ki86}. Of course, for a given Markov chain $\{ X_t \}$ the above representation is not unique. Still, in examples, the way in which Markov chains are defined is often through a recursive equation; all examples in \citet[][pp.~126--127]{Goldie91}, for instance, are of this type. The chain is stationary if and only if the random vectors $X_1 = \Phi(X_0, \eps_1)$ and $X_0$ are equal in law.

In \citet{Smith92} and \citet{Perfekt94}, excursions of a univariate Markov chain over a high threshold following an extreme event are shown to behave asymptotically and under quite general conditions as a (multiplicative) random walk. The theory has been extended to multivariate Markov chains in \citet{Perfekt97} and to higher-order Markov chains in \citet{Yun98, Yun00}. More recently, \citet{ReZe13} have analyzed the topic with a special view towards the convergence of Markov kernels and added a criterion to distinguish between extreme and non-extreme states of a Markov chain as the threshold rises. The random-walk representation is useful from a statistical perspective because it gives a handle on how to model the extremes of certain time series (\citet{BC00, CST97, STC97}). A useful, well-investigated class of processes for which the random walk structure is quite revealing are the stationary solutions to certain stochastic difference equations, including squared (generalized) autoregressive 
conditionally heteroskedastic (ARCH/GARCH) processes as a special case (\citet{BDM02b, GdHP04, dHRRdV89}). 

A limitation of the theory of \citet{Smith92}, \citet{Perfekt94} and \citet{ReZe13} is that it is specialized to univariate, nonnegative Markov chains. Similarly, \citet{Perfekt97} only considers the upper extremes of a multivariate Markov chain. When extending the theory to real-valued and higher dimensional chains, one has to keep in mind that extremes may be both positive or negative and that extreme values of $X_t$ may depend not only on $\|X_{t-1}\|$ but also on $X_{t-1}/ \|X_{t-1}\|$. The simplest case of the extension on which we will focus deals with real-valued univariate Markov chains, where an extreme value of $X_t$ may depend on the sign of $X_{t-1}$ as can be observed for instance in time series of logreturns of prices of financial securities in periods of high volatility. The observation of this so-called leverage effect has lead to the formulation of asymmetric extensions of GARCH models (cf., for example, \cite{Zi09}). For such Markov chains with tail switching potential, the 
random walk representation of excursions over high thresholds breaks down in the sense that the distribution of the multiplicative increment now depends in general 
on the sign of the chain on the previous step. In \citet{BC03}, a more general representation is postulated, involving in fact four transition mechanisms rather than one, corresponding to the four cases of transitions from and to upper or lower extreme states.


The novelty of this paper is two-fold: first, to explicitly state the random walk representation in the general $\mathbb{R}^d$-valued case; second, in the stationary case, to study the joint distribution of the forward and backward tail chain, coined the {\em back-and-forth tail chain}. Throughout, some remarkable simplifications in the (univariate) real-valued case will be studied in more detail. In particular, in the univariate case the backward tail chain is again a random walk which is in some sense dual the forward tail chain. Besides the assumption that the distribution of $X_0$ is regularly varying, the only condition is a relatively easy-to-check statement on the asymptotic behaviour of $\Phi(x, \, \cdot \,)$ for large $\|x\|$.

The outline of the paper is as follows. The forward tail chain of a possibly non-stationary $\mathbb{R}^d$-valued Markov chain is studied in section~\ref{S:forward}. For stationary Markov chains, the tail chain can be extended to the past of the process, the backward tail chain, see section~\ref{S:backforth}. Section~\ref{S:adjoint} describes a kind of adjoint relation between distributions which is motivated by a general property of tail processes of stationary processes. In section~\ref{S:BFTC}, we show that a certain class of processes, coined back-and-forth tail chains, which are derived from this adjoint distribution, form exactly the class of tail processes which arise in our Markovian setting. Finally, section~\ref{S:examples} provides some examples to the theory, including an application to stationary solutions of (multivariate) stochastic difference equations.

To conclude this section, let us fix some notations. We write $(x)_+=\max(x,0)$ for the positive part of $x \in\mathbb{R}$ and $(x)_-=\min(x,0)$ for the negative part. The transpose of a matrix $A$ is denoted by $A'$. The law of a random vector $X$ is denoted by $\law(X)$; weak convergence of probability measures is denoted by $\Rightarrow$. The probability measure degenerate at a point $x$ is denoted by $\delta_x$, and $\mbox{Unif}(E)$ denotes the uniform distribution on a compact set $E$. The indicator of an event $A$ is denoted by $\1_A(\cdot)$. We write $\overline{\mathbb{R}}$ for $\mathbb{R}\cup\{-\infty,\infty\}$, $\mathbb{S}^{d-1}$ for $\{x \in \mathbb{R}^d: \|x\|=1\}$ and $0$ for a vector (of suitable dimension) which consists of all zeros. Let $\ZZ$ be the set of integers and $\NN_0$ be the set of nonnegative integers.

\section{Forward tail chains}
\label{S:forward}

Let $X_0, X_1, X_2, \ldots$ be a homogeneous Markov chain as in \eqref{E:MC:1} and \eqref{E:MC:2}, not necessarily stationary. The focus of this section is on the weak limits of the finite-dimensional distributions of the process conditionally on $\|X_0\|$ being large (Theorem~\ref{T:forward}). Two conditions are required: Condition~\ref{C:RV} on the tails of $X_0$, and Condition~\ref{C:phi} on the asymptotics of $x \mapsto \Phi(x, e)$ for large $\|x\|$. See for instance \citet{Re07} for details on multivariate regular variation.

\begin{cond}
\label{C:RV}
The distribution of $X_0$ is multivariate regularly varying on $\overline{\mathbb{R}}^d \setminus \{0\}$, that is, there exists a non-degenerate probability measure $\Upsilon$ on $\mathbb{S}^{d-1}$ (called the spectral measure) and an $\alpha>0$ such that
\begin{equation}
\label{E:RV}
      \lim_{x \to \infty} \Pr\left(\|X_0\|>ux, \frac{X_0}{\|X_0\|} \in S\, \vrule \, \|X_0\| > x\right) = u^{-\alpha}\Upsilon(S) 
\end{equation}
for all Borel sets $S \subset \mathbb{S}^{d-1}$ which satisfy $\Upsilon(\partial S)=0$ and $u \geq 1$. \end{cond}

The second condition states that the function $\Phi$ in \eqref{E:MC:1} is asymptotically homogeneous in $x$ for large values of $\|x\|$.
\begin{cond}
\label{C:phi} 
There exists a measurable map $\phi: \mathbb{S}^{d-1} \times \mathbb{E} \mapsto \mathbb{R}^d$ such that, for all $e \in \mathbb{E}$,
\begin{equation}
\label{E:phi:V}
	\lim_{x \to \infty} x^{-1} \Phi(x s(x), e) = \phi(s,e)
\end{equation}
whenever $s(x) \to s$ in $\mathbb{S}^{d-1}$.

Moreover, if $\Pr(\phi(s, \eps_1) = 0) > 0$ for some $s \in \mathbb{S}^{d-1}$, then also $\Pr(\eps_1 \in \WW) = 1$, where $\WW$ is a measurable subset of $\EE$ such that for all $e \in \WW$,
\begin{equation}
\label{E:phi:W}
	\sup_{\|y\| \leq x} |\Phi(y, e)| = O(x), \qquad x \to \infty.
\end{equation}
\end{cond}

We extend the domain of the limit function $\phi$ in \eqref{E:phi:V} to $\mathbb{R}^d \times \mathbb{E}$ by setting
\begin{equation}
\label{eq:phi}
  \phi(v, e) = 
  \begin{cases}
    \| v \| \, \phi( v / \| v \|, e ) & \text{if $v \ne 0$,} \\
    0 & \text{if $v = 0$.}
  \end{cases}
\end{equation}

\begin{lem}\label{L:generalphi}
If Condition \ref{C:phi} holds, then 
\begin{equation}
\label{E:generalphi}
  \lim_{x \to \infty} x^{-1} \Phi(x v(x), e) 
  = \phi(v,e)
\end{equation}
whenever $v(x) \to v \in \mathbb{R}^d \setminus{0}$ and $e \in \mathbb{E}$. If $\Pr(\phi(s, \eps_1) = 0) > 0$ for some $s \in \mathbb{S}^{d-1}$, then \eqref{E:generalphi} also holds for $v(x) \to v = 0$ and $e \in \WW$.
\end{lem}
\begin{proof} If $v(x) \to v \in \mathbb{R}^d \setminus{0}$, then both $\|v(x)\|\to \|v\|$ and $v(x)/\|v(x)\| \to v/\|v\|$.  Thus
\[
  \lim_{x \to \infty} \frac{\Phi(x v(x), e)}{x} 
  = \lim_{x \to \infty} \|v(x)\|\frac{\Phi(x \|v(x)\|(v(x)/\|v(x)\|), e)}{x\|v(x)\|}
  = \| v \| \, \phi( v / \| v \|, e )
\]
which, by \eqref{eq:phi}, gives \eqref{E:generalphi}. The case $v(x) \to 0$  follows from \eqref{E:phi:W}.
\end{proof}

\begin{thm}
\label{T:forward}
Let $\{ X_t : t \in \mathbb{N}_0 \}$ be given by \eqref{E:MC:1}--\eqref{E:MC:2}. If Conditions~\ref{C:RV} and \ref{C:phi} hold, then for every integer $t \geq 0$, as $x \to \infty$,
\begin{equation}
\label{E:forward:1}
	\law \biggl( \frac{\|X_0\|}{x}, \frac{X_0}{\|X_0\|}, \frac{X_1}{\|X_0\|}, \ldots, \frac{X_t}{\|X_0\|} \bigg| \|X_0\| > x \biggr)
	\Rightarrow \law(Y, M_0, M_1, \ldots, M_t)
\end{equation}
with 
\begin{equation}
\label{E:Mj}
	M_j =  \phi(M_{j-1},\eps_j), \qquad j = 1, 2, \ldots,
\end{equation}
and
\begin{equation}
\label{E:forward:2}
\mbox{\begin{minipage}[h]{0.90\textwidth}
\begin{itemize}
\item[(i)] $Y, M_0, \eps_1, \eps_2, \ldots$ are independent with $\eps_t$ as in \eqref{E:MC:2}(i);
\item[(ii)] $\Pr(Y > y) = y^{-\alpha}$ for $y \geq 1$;
\item[(iii)] $\mathcal{L}(M_0)=\Upsilon$.
\end{itemize}
\end{minipage}}
\end{equation}
We call $\{M_t: t \in \mathbb{N}_0\}$ the \emph{forward tail chain of} $\{X_t: t \in \mathbb{N}_0\}$.
\end{thm}

\begin{proof}
The argument is by induction on $t$. The case $t = 0$ is a straightforward consequence of Condition~\ref{C:RV}. So let $t$ be a positive integer and let $f : \RR\times(\RR^d)^{t+1} \to \RR$ be bounded and continuous. We have to show that
\begin{equation}
\label{E:forward:10}
	\lim_{x \to \infty}
	\EV \biggl[
	f \biggl( \frac{\|X_0\|}{x}, \frac{X_0}{\|X_0\|}, \ldots, \frac{X_t}{\|X_0\|} \biggr)
	\, \biggl| \, \|X_0\| > x
	\biggr]
	= \EV [ f(Y, M_0, \ldots, M_t) ].
\end{equation}
By \eqref{E:MC:1}, if $X_0 \neq 0$,
\[
	\frac{X_t}{\|X_0\|}
	= \frac{\Phi(X_{t-1}, \eps_t)}{\|X_0\|}
	= \frac{\Phi(x \frac{\|X_0\|}{x} \frac{X_{t-1}}{\|X_0\|}, \eps_t)}
	{x \frac{\|X_0\|}{x}}.
\]
Hence,
\begin{eqnarray}
\label{E:forward:20}
	\lefteqn{
	\EV \biggl[
	f \biggl( \frac{\|X_0\|}{x}, \frac{X_0}{\|X_0\|}, \ldots, \frac{X_t}{\|X_0\|} \biggr)
	\, \biggl| \, \|X_0\| > x
	\biggr]
	} \\
	&=&
	\EV \biggl[
	g_x \biggl( \frac{\|X_0\|}{x}, \frac{X_0}{\|X_0\|}, \ldots, \frac{X_{t-1}}{\|X_0\|} \biggr)
	\, \biggl| \, \|X_0\| > x
	\biggr] 
	\nonumber
\end{eqnarray}
where
\begin{equation}
\label{E:forward:gx}
	g_x(y, x_0, \ldots, x_{t-1})
	= \EV \biggl[ f \biggl( y, x_0, \ldots, x_{t-1}, \frac{\Phi(xyx_{t-1}, \eps_t)}{xy} \biggr) 
	\biggr]
\end{equation}
(note that the expectation is taken with respect to the distribution of $\epsilon_t$).
Define
\begin{equation}
\label{E:forward:g}
	g(y, x_0, \ldots, x_{t-1})
	= \EV [ f (y, x_0, \ldots, x_{t-1}, \phi(x_{t-1},\eps_t) ) ].
\end{equation}
By \eqref{E:Mj},
\begin{equation}
\label{E:forward:40}
	\EV [ f(Y, M_0, \ldots, M_t) ] = \EV [ g(Y, M_0, \ldots, M_{t-1}) ].
\end{equation}
In view of the identities \eqref{E:forward:20} and \eqref{E:forward:40}, the limit relation in \eqref{E:forward:10} will follow if we can show that
\begin{equation}
\label{E:forward:50}
	\EV \biggl[
	g_x \biggl( \frac{\|X_0\|}{x}, \frac{X_0}{\|X_0\|}, \ldots, \frac{X_{t-1}}{\|X_0\|} \biggr)
	\, \biggl| \, \|X_0\| > x
	\biggr]
	\to \EV [ g(Y, M_0, \ldots, M_{t-1}) ]
\end{equation}
as $x \to \infty$. In turn, \eqref{E:forward:50} will follow from the induction hypothesis and an extension of the continuous mapping theorem \citep[][Theorem~18.11]{vdV98} provided
\begin{equation}
\label{E:forward:55}
	\lim_{x \to \infty} g_x( y(x), x_0(x), \ldots, x_{t-1}(x) ) = g( y, x_0, \ldots, x_{t-1} )
\end{equation}
whenever $y(x) \to y$ and $x_i(x) \to x_i$ as $x \to \infty$ with $(y, x_0, \ldots, x_{t-1})$ ranging over a set $E \subset \RR \times (\RR^d)^t$ with $\Pr((Y, M_0, \ldots, M_{t-1}) \in E) = 1$. From the definitions of $g_x$ and $g$ in \eqref{E:forward:gx} and \eqref{E:forward:g}, respectively, equation~\eqref{E:forward:55} is implied by
\begin{equation}
\label{E:forward:60}
	\lim_{x \to \infty} \frac{\Phi(x w(x), v)}{x} = \phi(w,v)
\end{equation}
whenever $\lim_{x \to \infty} w(x) = w$ and where $w$ and $v$ range over sets that receive probability one by the distributions of $M_{t-1}$ and $\eps_1$, respectively. Since \eqref{E:forward:60} is ensured by Condition~\ref{C:phi} and Lemma \ref{L:generalphi}, the statement follows. 
\end{proof}

\section{Backward tail processes}
\label{S:backforth}

From now on, the process $\{ X_t \}$ in \eqref{E:MC:1} and \eqref{E:MC:2} is assumed to be strictly stationary. A necessary and sufficient condition for stationarity is that 
\begin{equation}
\label{E:MC:3}
	\law(\Phi(X_0, \eps_1)) = \law(X_0).
\end{equation}
It may be highly non-trivial to find the law for $X_0$ that solves \eqref{E:MC:3}. But even when the stationary distribution does not admit an explicit expression, its tails may in many cases be found by the theory developed originally in \citet{Kesten73}, \citet{Letac86} and \citet{Goldie91}. For recent results on specific models, see for instance \citet{KluPer03, KluPer04}, \citet{DeSa04}, \citet{Mirek11}, \citet{Bura12}, and \citet{CV13}.

If the process $\{ X_t \}$ is stationary, then by Kolmogorov's extension theorem and changing the probability space if necessary, the range of $t$ can without loss of generality be assumed to be the set of all integers, $\ZZ$; recall that we are interested in distributional properties only, not in almost sure properties, for instance.

Our aim is to extend Theorem~\ref{T:forward} and find the asymptotic distribution of the random vector $(X_{-s}, \ldots, X_t)$ conditionally on $\|X_0\| > x$ as $\| x \| \to \infty$, for all integer $s$ and $t$ (Corollary~\ref{Cor:spectralisBFTC}). 
According to \citet[][Theorem~2.1]{BS09}, if the underlying process is stationary, the existence of a forward tail process $(t \in \NN_0)$ is enough to guarantee the existence of the tail process as a whole ($t \in \mathbb{Z}$).

\begin{prop}
\label{P:BS09:2.1}
Let $\{ X_t : t \in \ZZ\}$ be a stationary Markov chain with distribution determined by \eqref{E:MC:1}, \eqref{E:MC:2} and \eqref{E:MC:3}. If Conditions~\ref{C:RV} and \ref{C:phi} hold, then there exists a process $\{M_t: t \in \mathbb{Z}\}$ such that
\begin{equation}
\label{E:forwardandbackward:existence}
	\law \biggl(\frac{X_{-s}}{\|X_0\|}, \dots, \frac{X_0}{\|X_0\|}, \ldots, \frac{X_t}{\|X_0\|} \bigg| \|X_0\| > x \biggr)
	\Rightarrow \law(M_{-s}, \ldots, M_0, \ldots, M_t)
\end{equation}
for all integer $s,t \geq 0$.
\end{prop}

\begin{proof}
This follows from our Theorem \ref{T:forward} and Theorem 2.1 in \citet{BS09}, combined with a continuous mapping argument.
\end{proof}

We call the process $\{M_t: t \in \mathbb{Z}\}$ the \emph{spectral (tail) process} of $\{X_t: t \in \mathbb{Z}\}$, in accordance with the definition of the process $\{\Theta_t: t \in \mathbb{Z}\}$ in \citet{BS09}. 

\citet{BS09} also state an important property of the limiting process.
\begin{prop}
\label{P:BS09:3.1} 
Let $\{ X_t : t \in \ZZ\}$ be a stationary Markov chain with distribution determined by \eqref{E:MC:1}, \eqref{E:MC:2} and \eqref{E:MC:3} and spectral process $\{M_t: t \in \mathbb{Z}\}$. Then for all $s,t \geq 0$ and for all bounded and measurable $f:(\mathbb{R}^d)^{s+t+1} \to \mathbb{R}$ satisfying $f(y_{-s}, \ldots, y_t)=0$ whenever $y_{-s}=0$,
\begin{equation}
\label{E:timechange}
 E\left[f(M_{-s}, \ldots, M_{t}) \right]= E\left[f\left(\frac{M_0}{\|M_s\|}, \ldots, \frac{M_{s+t}}{\|M_s\|} \right) \|M_s\|^\alpha\1_{\{M_s \neq 0\}} \right].
\end{equation}
\end{prop}
\begin{proof}
It follows directly from our Proposition \ref{P:BS09:2.1} and Theorem 3.1 in \cite{BS09} that 
\begin{equation}\label{TC1}
 E\left[f(M_{-s-i}, \ldots, M_{t-i}) \right]= E\left[f\left(\frac{M_{-s}}{\|M_i\|}, \ldots, \frac{M_{t}}{\|M_i\|} \right) \|M_i\|^\alpha\1_{\{M_i \neq 0\}} \right].
\end{equation}
holds for all bounded and continuous $f:(\mathbb{R}^d)^{t+s+1} \to \mathbb{R}$ satisfying $f(y_{-s}, \ldots, y_t)=0$ whenever $y_0=0$ (instead of $y_{-s}=0$) and all $i \in \mathbb{Z}$. We have added the indicator function on the right-hand side for greater clarity. Let $s, t$ and $f$ be as in the statement of the Proposition. Apply \eqref{TC1} to the indices $(\underline{s}, \underline{t}, \underline{i}) = (0, t+s, s)$ to arrive at \eqref{E:timechange}; note that $\underline{s} + 1 + \underline{t} = s + 1 + t$ and that $f(x_{-\underline{s}}, \ldots, x_{\underline{t}}) = 0$ as soon as $x_0 = 0$. Thus, for functions $f$ which are additionally assumed to be continuous, the statement follows directly. 

For the general case, set for abbreviation $\mathbb{A}^\ast:=(\mathbb{R}^d)^{s+t+1}\setminus(\{0\} \times (\mathbb{R}^d)^{s+t})$. Furthermore, let $\mu$ denote the restriction of the law of $(M_{-s}, \ldots, M_t)$ to $\mathbb{A}^\ast$ and let $\nu$ denote the measure on $\mathbb{A}^\ast$ defined by 
$$ \nu(f)=E\left[f\left(\frac{M_{-s}}{\|M_i\|}, \ldots, \frac{M_{t}}{\|M_i\|} \right) \|M_i\|^\alpha\1_{\{M_i \neq 0\}} \right] $$
for all bounded and continuous $f$ on $\mathbb{A}^\ast$. In order to show \eqref{E:timechange} for general bounded and measurable $f$ with $f(y_{-s}, \ldots, y_t)=0$ if $y_{-s}=0$ it suffices to show that $\mu$ and $\nu$ coincide. The closed sets of $(\mathbb{R}^d)^{s+t+1}$ which are bounded away from $\{0\} \times (\mathbb{R}^d)^{s+t}$ are a $\pi$-system generating $\mathbb{B}(\mathbb{A}^\ast)$. Indicator functions of closed sets $A$ can be written as pointwise limits of continuous functions with values in $[0,1]$. If $A$ is bounded away from $\{0\} \times (\mathbb{R}^d)^{s+t}$ we can choose these approximating continuous functions in such a way that they vanish on $\{0\} \times (\mathbb{R}^d)^{s+t}$. Thus, by dominated convergence $\mu(A)=\nu(A)$ for all sets $A$ of a generating $\pi$-system and therefore $\mu=\nu$ on the Borel sets of $\mathbb{A}^\ast$ \citep[][Theorem~2.2]{Bi68}, which finishes the proof.
\end{proof}

By Lemma~2.2 in \citet{BS09} it follows that the distribution of $\{M_t: t \in \ZZ\}$ is uniquely determined by the distribution of $\{M_t: t \in \NN_0\}$ (and $\alpha>0$). We will use \eqref{E:timechange} to analyze the structure of the spectral process with a special focus on the backward process $\{M_ {-t}: t \in \NN_0\}$. At the heart of the connection between the forward and backward processes is an adjoint relation between the laws of $(M_0,M_1)$ and $(M_0,M_{-1})$, studied next.

\section{An adjoint relation between distributions}
\label{S:adjoint}

A special case of the equality \eqref{E:timechange} is
\begin{equation}\label{timechange:onestep}
\EV\left[f(M_{-1}, M_0)\right]= \EV\left[f\left(\frac{M_0}{\|M_1\|}, \frac{M_1}{\|M_1\|} \right) \|M_1\|^\alpha\1_{\{M_1 \neq 0\}} \right]
\end{equation}
for all $f:(\mathbb{R}^d)^2 \to \mathbb{R}$ satisfying $f(y_0, y_1)=0$ whenever $y_0=0$. Starting from a given distribution of $(M_0, M_1)$ we will in the following characterize the distributions of $(M_{-1}, M_0)$ which satisfy \eqref{timechange:onestep}. 
For such an adjoint distribution to exist,
the distribution $(M_0, M_1)$ cannot be chosen arbitrarily from the distributions on $\mathbb{S}^{d-1}\times \mathbb{R}^d$. We therefore introduce the following set of ``admissible'' distributions.

\begin{defn}
\label{def:admissible}
For $\alpha \in (0, \infty)$, let $\mathcal{M}_\alpha = \mathcal{M}_{\alpha,d}$ be the set of all probability measures $P$ on $\mathbb{S}^{d-1} \times \RR^d$ such that
\begin{equation}
\label{eq:admissible}
  \int_{\mathbb{S}^{d-1}\times (\RR^d \setminus \{ 0 \})} \1_S( m / \norm{m} ) \, \norm{m}^\alpha \, P(\ds, \dm) \le P( S \times \RR^d )
\end{equation}
for every Borel set $S \subset \mathbb{S}^{d-1}$. We call $\mathcal{M}_\alpha$ the set of \emph{admissible distributions} for $\alpha>0$. 
\end{defn}

Note that for $P \in \mathcal{M}_\alpha$ we have
\[
  \int_{\mathbb{S}^{d-1} \times \RR^d} \norm{m}^\alpha \, P(\ds, \dm) \le 1.
\]
We now make the already mentioned notion of an ``adjoint'' distribution more concise. 

\begin{defn}
\label{def:adjoint}
For $P \in \mathcal{M}_\alpha$, define a signed Borel measure $P^*$ on $\mathbb{S}^{d-1} \times \RR^d$ by
\begin{align}
\label{eq:adjoint:S0}
  P^*(S \times \{ 0 \}) &= P(S \times \RR^d) - \int_{\mathbb{S}^{d-1} \times (\RR^d \setminus \{ 0 \})} \1_S( m / \norm{m} ) \, \|m\|^\alpha \, P(\ds, \dm), \\
\label{eq:adjoint:E}
  P^*(E) &= \int_{ \mathbb{S}^{d-1} \times (\RR^d \setminus \{ 0 \}) } \1_E( m / \norm{m}, s / \norm{m} ) \, \norm{m}^\alpha \, P(\ds, \dm),
\end{align}
for Borel sets $S \subset \mathbb{S}^{d-1}$ and $E \subset \mathbb{S}^{d-1} \times (\RR^d \setminus \{ 0 \})$.
We call $P^*$ the \emph{adjoint measure of $P$ in $\mathcal{M}_\alpha$}.
\end{defn}

\begin{lem}
\label{lem:adjoint}
Let $P \in \mathcal{M}_\alpha$ and let $P^*$ be as in Definition~\ref{def:adjoint}.
 \begin{enumerate}[(i)]
  \item $P^*$ is a probability measure and the marginal distributions induced by $P$ and $P^*$ on $\mathbb{S}^{d-1}$ are the same.
\item For every measurable function $f : \mathbb{S}^{d-1} \times (\RR^d \setminus \{ 0 \}) \to \RR$,
\begin{multline}
\label{eq:adjoint:f}
  \int_{ \mathbb{S}^{d-1} \times (\RR^d \setminus \{ 0 \}) } f(s^*, m^*) \, P^*(\ds^*, \dm^*) \\*
  = \int_{ \mathbb{S}^{d-1} \times (\RR^d \setminus \{ 0 \}) } f( m / \norm{m}, s / \norm{m} ) \, \norm{m}^\alpha \, P(\ds, \dm)
\end{multline}
in the sense that if one integral exists, then so does the other, and they are the same.
\item $P^* \in \mathcal{M}_\alpha$.
\item $(P^*)^* = P$.
  \end{enumerate}
\end{lem}

\begin{proof}
(i) By \eqref{eq:admissible}, $P^*$ is a nonnegative Borel measure. Let $S$ be a Borel subset of $\mathbb{S}^{d-1}$. We have
\[
  P^*( S \times \RR^d ) = P^*( S \times \{ 0 \} ) + P^*\bigl( S \times (\RR^d \setminus \{ 0 \}) \bigr).
\]
Applying \eqref{eq:adjoint:S0} to the first term on the right-hand side and applying \eqref{eq:adjoint:E} with $E = S \times (\RR^d \setminus \{ 0 \})$ to the second term on the right-hand side yields
\[
  P^*( S \times \RR^d ) = P( S \times \RR^d ).
\]
It follows that $P^*$ is a probability measure (take $S = \mathbb{S}^{d-1}$) on $\mathbb{S}^{d-1} \times \RR^d$ inducing the same marginal distribution on $\mathbb{S}^{d-1}$ as $P$.

(ii) By \eqref{eq:adjoint:E}, equation~\eqref{eq:adjoint:f} holds for indicator functions $\1_E$ of Borel subsets $E$ of $\mathbb{S}^{d-1} \times (\RR^d \setminus \{ 0 \})$. The extension to general bounded, measurable functions follows from the definition of the integral.

(iii) Let $S$ be a Borel subset of $\mathbb{S}^{d-1}$. We will apply \eqref{eq:adjoint:f} to the function
\[
  f(s, m) = \1_S( m / \norm{m} ) \, \norm{m}^\alpha \qquad \text{for $(s, m) \in \mathbb{S}^{d-1} \times (\RR^d \setminus \{ 0 \})$}.
\]
We find
\begin{align*}
  \lefteqn{
  \int_{ \mathbb{S}^{d-1} \times (\RR^d \setminus \{ 0 \}) } \1_S( m^* / \norm{m^*} ) \, \norm{m^*}^\alpha \, P^*(\ds^*, \dm^*)
  } \\
  &= \int_{ \mathbb{S}^{d-1} \times (\RR^d \setminus \{ 0 \}) } f( s^*, m^* ) \, P^*(\ds^*, \dm^*) \\
  &= \int_{ \mathbb{S}^{d-1} \times (\RR^d \setminus \{ 0 \}) } f( m / \norm{m}, s / \norm{m} ) \, \norm{m}^\alpha \, P(\ds, \dm) \\
  &= \int_{ \mathbb{S}^{d-1} \times (\RR^d \setminus \{ 0 \}) } \1_S \biggl( \frac{s / \norm{m}}{\norm{(s / \norm{m})}} \biggr) \, \norm{(s / \norm{m})}^\alpha \, \norm{m}^\alpha \, P(\ds, \dm) \\
  &= P \bigl( S \times ( \RR^d \setminus \{ 0 \} ) \bigr) \\
  &\le P ( S \times \RR^d ) = P^* ( S \times \RR^d ),
\end{align*}
where we applied (i) in the last step.

(iv) Let $Q = (P^*)^*$. We already know that $Q$ is a probability measure on $\mathbb{S}^{d-1} \times \RR^d$, that $Q \in \mathcal{M}_\alpha$, and that the marginal induced by $Q$ on $\mathbb{S}^{d-1}$ coincides with the one of $P^*$ and thus with the one of $P$. Let $f$ be a nonnegative, measurable function on $\mathbb{S}^{d-1} \times (\RR^d \setminus\{ 0 \})$. Define the nonnegative, measurable function $g$ on $\mathbb{S}^{d-1} \times (\RR^d \setminus \{ 0 \})$ by
\[
  g( s, m ) = f( m / \norm{m}, s / \norm{m} ) \, \norm{m}^\alpha, \qquad \text{for $(s, m) \in \mathbb{S}^{d-1} \times (\RR^d\setminus\{0\})$}.
\]
We have
\begin{multline}
\label{eq:adjoint:g2f}
  g( m / \norm{m}, s / \norm{m} ) \, \norm{m}^\alpha \\
  = f\biggl( \frac{s / \norm{m}}{\norm{(s / \norm{m})}}, \frac{m / \norm{m}}{\norm{(s / \norm{m})}} \biggr) \, \norm{(s / \norm{m})}^\alpha \, \norm{m}^\alpha
  = f(s, m).
\end{multline}
By \eqref{eq:adjoint:f} applied first to $Q$ and $f$ and then to $P^*$ and $g$, we have
\begin{align*}
  \lefteqn{
  \int_{ \mathbb{S}^{d-1} \times (\RR^d \setminus \{ 0 \}) } f( s, m ) \, Q(\ds, \dm)
  } \\
  &= \int_{ \mathbb{S}^{d-1} \times (\RR^d \setminus \{ 0 \}) } f( m^* / \norm{m^*}, s^* / \norm{m^*} ) \, \norm{m^*}^\alpha \, P^*(\ds^*, \dm^*) \\
  &= \int_{ \mathbb{S}^{d-1} \times (\RR^d \setminus \{ 0 \}) } g( s^*, m^* ) \, P^*(\ds^*, \dm^*) \\
  &= \int_{ \mathbb{S}^{d-1} \times (\RR^d \setminus \{ 0 \}) } g( m / \norm{m}, s / \norm{m} ) \, \norm{m}^\alpha \, P(\ds, \dm) \\
  &= \int_{ \mathbb{S}^{d-1} \times (\RR^d \setminus \{ 0 \}) } f(s, m) \, P(\ds, \dm),
\end{align*}
where we used \eqref{eq:adjoint:g2f} in the last step. It follows that $Q$ and $P$ coincide on $\mathbb{S}^{d-1} \times (\RR^d \setminus \{ 0 \})$. As $Q$ and $P$ also induce the same marginal distributions on $\mathbb{S}^{d-1}$, it follows that they must also coincide on $\mathbb{S}^{d-1} \times \{ 0 \}$. As a consequence, $Q$ is equal to $P$.
\end{proof} 

The next lemma shows that the class $\mathcal{M}_\alpha$ and the adjoint relation on it arise naturally in the context of regularly varying Markov chains.

\begin{lem}
\label{L:M1}
Let $\{ X_t : t \in \ZZ\}$ be a stationary Markov chain with distribution determined by \eqref{E:MC:1}, \eqref{E:MC:2} and \eqref{E:MC:3}. If Conditions~\ref{C:RV} and \ref{C:phi} hold, then $\law(M_0, M_1)$ belongs to $\mathcal{M}_\alpha$ and its adjoint is equal to $\law(M_0, M_{-1})$.
\end{lem}

\begin{proof}
To prove admissibility, we have to show that
\begin{equation}
\label{eq:markov-admissible:S}
  \EV[ \1_S( M_1 / \norm{M_1} ) \, \norm{M_1}^\alpha ] \le \Pr( M_0 \in S )
\end{equation}
for every Borel set $S \subset \mathbb{S}^{d-1}$. Let $f$ be a bounded, nonnegative and continuous function on $\mathbb{S}^{d-1}$. We will show that
\begin{equation}
\label{eq:markov-admissible:f}
  \EV[ f( M_1 / \norm{M_1} ) \, \norm{M_1}^\alpha ] \le \EV[ f(M_0) ].
\end{equation}
Equation~\eqref{eq:markov-admissible:f} implies \eqref{eq:markov-admissible:S} for closed sets $S$ because the indicator function of a closed set $S$ can be written as the pointwise limit of a decreasing sequence of continuous functions taking values in the interval $[0, 1]$. From this we arrive at \eqref{eq:markov-admissible:S} for an arbitrary Borel set $S$ by invoking an increasing sequence of closed sets $S_n$ contained in $S$ such that $\EV[ \1_{S_n}( M_1 / \norm{M_1} ) \, \norm{M_1}^\alpha ]$ and $\Pr( M_0 \in S_n )$ converge to $\EV[ \1_S( M_1 / \norm{M_1} ) \, \norm{M_1}^\alpha ]$ and $\Pr( M_0 \in S )$ respectively; see for instance Theorem~1.1 on p.~7 in \citet{Bi68}.

Let $\delta > 0$. By stationarity of $\{X_t:t \in \ZZ\}$ and by definition of the spectral process $\{M_t:t \in \ZZ\}$, we have
\begin{align*}
  \EV[ f(M_0) ]
  &= \lim_{x \to \infty} \EV[ f( X_1 / \norm{X_1} ) \mid \norm{X_1} > x ] \\
  &\ge \limsup_{x \to \infty} 
  \EV[ \1_{ \{ \norm{X_0} > \delta x \} } \, f( X_1 / \norm{X_1} ) \mid \norm{X_1} > x ] \\
  &= \limsup_{x \to \infty} 
  \frac{\Pr[ \norm{X_0} > \delta x ]}{\Pr[ \norm{X_1} > x]} \, \EV[ f( X_1 / \norm{X_1} ) \, \1_{ \{ \norm{X_1} > x \} } \mid \norm{X_0} > \delta x ] \\
  &= \delta^{-\alpha} \EV[ f( M_1 / \norm{M_1} ) \, \1_{ \{ Y \norm{M_1} > \delta^{-1} \} } ].
\end{align*}
In the last line, $Y$ is a Pareto($\alpha$) random variable, independent of $M_1$. As $\Pr( Y \norm{M_1} = \delta^{-1} ) = 0$ by continuity of the law of $Y$, the last equality in the above display follows from the continuous mapping theorem.

Since the distribution of $Y^{-\alpha}$ is uniform on the interval $(0, 1)$, we have
\begin{eqnarray*}
 \delta^{-\alpha} \EV[ f( M_1 / \norm{M_1} ) \, \1_{ \{ Y \norm{M_1} > \delta^{-1} \} } ] &=& \delta^{-\alpha} \EV[ f( M_1 / \norm{M_1} ) \, \1_{ \{ \delta^\alpha \norm{M_1}^\alpha > Y^{-\alpha} \} } ] \\
  &=& \delta^{-\alpha} \EV[\EV[ f( M_1 / \norm{M_1} ) \, \1_{ \{ \delta^\alpha \norm{M_1}^\alpha > Y^{-\alpha} \} }|M_1] ] \\
  &=& \delta^{-\alpha} \EV[ f( M_1 / \norm{M_1} ) \, \min( \delta^\alpha \norm{M_1}^\alpha, 1 ) ] \\
  &=& \EV[ f(M_1 / \norm{M_1} ) \, \min( \norm{M_1}^\alpha, \delta^{-\alpha} ) ].
\end{eqnarray*}
We obtain that for every $\delta > 0$,
\[
  \EV[ f(M_0) ] \ge \EV[ f(M_1 / \norm{M_1} ) \, \min( \norm{M_1}^\alpha, \delta^{-\alpha} ) ].
\]
Take the limit as $\delta \to 0$ and apply the monotone convergence theorem to obtain \eqref{eq:markov-admissible:f}. 

Next we show that the adjoint of $\law(M_0, M_1)$ is equal to $\law(M_0, M_{-1})$. We have to check the two equations
\begin{align}
\nonumber
  \Pr((M_0,M_{-1}) \in S \times \{0\})
  &=\Pr(M_0 \in S)-\EV[\1_{\mathbb{R}^d \setminus\{0\}}(M_1)\1_S(M_1/\|M_1\|)\|M_1\|^\alpha], \\
\label{Eq:lemM2:2}
  \Pr((M_0,M_{-1}) \in E)
  &=\EV[\1_{\mathbb{R}^d \setminus\{0\}}(M_1)\1_E(M_1/\|M_1\|,M_0/\|M_1\|)\|M_1\|^\alpha],
\end{align}
for all Borel sets $S \subset \mathbb{S}^{d-1}$ and $E \subset \mathbb{S}^{d-1}\times (\mathbb{R}^d\setminus\{0\})$. Since the first component $M_0$ is common to both laws, it is sufficient to check only the second equation, \eqref{Eq:lemM2:2}.

Set $f(m_{-1},m_0)=\1_E(m_0,m_{-1})$ on $\mathbb{R} \times \mathbb{S}^{d-1}$. Note that $f(0,m_0)=0$. Apply equation \eqref{timechange:onestep} to $f$:
\begin{align*} 
  \Pr((M_0,M_{-1}) \in E)
  &= \EV[f(M_{-1},M_0)] \\
  &= \EV[f(M_0/\|M_1\|,M_1/\|M_1\|)\|M_1\|^\alpha\1_{\{M_1 \neq 0\}}] \\
  &= \EV[\1_{\mathbb{R}^d \setminus\{0\}}(M_1)\1_E(M_1/\|M_1\|,M_0/\|M_1\|)\|M_1\|^\alpha],
\end{align*}
which gives \eqref{Eq:lemM2:2}, as required. 
\end{proof}

\begin{rem}\label{partsimple} The determination of the adjoint measure is particularly simple for probability measures $P$ such that 
\begin{equation}
\label{eq:partsimple}
  \int_{\mathbb{S}^{d-1} \times \mathbb{R}^d}\|m\|^\alpha P(\ds, \dm)=1,  
\end{equation}
since in this case $P^*(\mathbb{S}^{d-1} \times \{ 0 \})=0$ by \eqref{eq:adjoint:S0} and $P^*$ is completely described by \eqref{eq:adjoint:E}.
\end{rem}

\begin{rem}\label{selfadjoint}
We call a measure $P \in \mathcal{M}_\alpha$ \emph{self-adjoint} if $P^\ast=P$. An example for such a distribution in the case of $d=1$ and $\alpha=1$ is given by $P=\mathcal{L}(1,Y)$, where $Y=\exp(X-1/2)$ for standard normally distributed $X$ (cf.\ Example~3.2 in \cite{S07}).
 \end{rem}

Definition~\ref{def:adjoint} and Lemma~\ref{lem:adjoint} generalize Proposition~3.1 in \cite{S07} to the multivariate case. Examples~3.2--3.4 in the latter reference illustrate the adjoint relation for laws on $\{-1, +1\} \times \mathbb{R}$. We conclude the section with a multivariate example.

\begin{ex}
\label{Ex:mult:1}
Let $\alpha > 0$ and let $P$ be the law of $(C, R Q C)$ with $C$, $R$ and $Q$ independent, $C$ taking values in $\mathbb{S}^{d-1}$, $R$ a positive random variable with $\EV[R^\alpha] = 1$, and $Q$ a random orthogonal $d \times d$ matrix, that is $Q' = Q^{-1}$ a.s.; also assume that the laws of $C$ and $QC$ are the same (cf.\ also Example \ref{Ex:KestenRDE}). One verifies easily that $P \in \mathcal{M}_\alpha$ and that \eqref{eq:partsimple} holds, so that the adjoint law $P^*$ is concentrated on $\mathbb{S}^{d-1} \times (\mathbb{R}^d \setminus \{ 0 \})$. It may thus be derived from \eqref{eq:adjoint:E} that for Borel sets $S \subset \mathbb{S}^{d-1}$ and $T \subset \mathbb{R}^d \setminus \{ 0 \}$,
\begin{equation}
\label{E:Ex:mult:1}
  P^*(S \times T)=\EV[\1_S(QC)\1_{T}(C/R)R^\alpha].
\end{equation}
If we assume in addition that $C$ is uniformly distributed on $\mathbb{S}^{d-1}$ (which readily implies $\mathcal{L}(C)=\mathcal{L}(QC)$ for any law of $Q$), then
\begin{eqnarray*}
\EV[\1_S(QC)\1_{T}(C)]&=&\EV\left[\int_{\mathbb{R}^{d \times d}} \1_S(qC)\1_{T}(C)P^Q(dq)\right] \\
&=& \EV\left[\int_{\mathbb{R}^{d \times d}} \1_S(C)\1_{T}(q'C)P^Q(dq)\right] \\
&=& \EV[\1_S(C)\1_{T}(Q'C)]
\end{eqnarray*}
and it follows from \eqref{E:Ex:mult:1} that $P^*$ is the law of $(C^*, R^* Q^* C^*)$, with $C^*$, $R^*$ and $Q^*$ independent, $\law(C^*) = \law(C)$, $\law(Q^*) = \law(Q')$, and the law of $R^* > 0$ given by $\EV[f(R^*)] = \EV[f(1/R) \, R^\alpha]$ for measurable functions $f$ on $(0, \infty)$.
\end{ex}

\section{Back-and-forth tail chains  and the spectral process}
\label{S:BFTC}

In this section, we will analyze a certain class of discrete-time processes which are constructed from a pair of adjoint distributions. We will see that this class of processes fulfills equation \eqref{E:timechange} for all $i,s,t \in \mathbb{Z}$ with $s \leq 0 \leq t$. 

\begin{defn}
\label{D:BFTC}
A $d$-dimensional discrete-time process $\{M_t:t \in \ZZ\}$ is called a \emph{back-and-forth tail chain} with index $\alpha>0$, notation $\BFTC(\alpha)$, if the following properties hold:
\begin{enumerate}[(i)]
\item $\law(M_0, M_1)$ and $\law(M_0, M_{-1})$ belong to $\mathcal{M}_\alpha$ and are adjoint;
\item the forward process $\{M_t:t \in \NN_0\}$ is a Markov chain with respect to the filtration \linebreak $\sigma( M_s, -\infty < s \le t)$, $t \ge 0$, and the Markov kernel satisfies
\begin{multline*}
  \Pr( M_t \in \,\cdot\, \mid M_{t-1} = x_{t-1}) \\
  = \begin{cases} 
      \delta_0(\,\cdot\,) & \text{if $x_{t-1} = 0$,} \\
      \Pr( \norm{x_{t-1}} M_1 \in \,\cdot\, \mid M_0 = x_{t-1}/\norm{x_{t-1}} ) & \text{if $x_{t-1} \neq 0$;}
    \end{cases}
\end{multline*}
\item the backward process $\{M_{-t}:t \in \NN_0\}$ is a Markov chain with respect to the filtration $\sigma( M_{-s}, -\infty < s \le t)$, $t \ge 0$, and the Markov kernel satisfies
\begin{multline*}
  \Pr( M_{-t} \in \,\cdot\, \mid M_{-t+1} = x_{-t+1}) \\
  = \begin{cases} 
      \delta_0(\,\cdot\,) & \text{if $x_{-t+1} = 0$,} \\
      \Pr( \norm{x_{-t+1}} M_{-1} \in \,\cdot\, \mid M_0 = x_{-t+1}/\norm{x_{-t+1}} ) & \text{if $x_{-t+1} \neq 0$.}
    \end{cases}
\end{multline*}
\end{enumerate}
\end{defn}
 
Clearly, $\{ M_t : t \in \ZZ \}$ is a $\BFTC(\alpha)$ if and only if $\{ M_{-t} : t \in \ZZ \}$ is a $\BFTC(\alpha)$. The distribution of a BFTC($\alpha$) is completely determined by an admissible law of $(M_0,M_1)$ (and $\alpha>0$). 


The fact that the distributions $P=\mathcal{L}(M_0, M_1)$ and $P^\ast=\mathcal{L}(M_0, M_{-1})$ are adjoint in $\mathcal{M}_\alpha$ implies that for every measurable function $f : \RR^d \times \mathbb{S}^{d-1} \to \RR$ such that $f(0, s) = 0$ for all $s \in \mathbb{S}^{d-1}$, we have
 \begin{align}
   \EV [ f ( M_{-1}, M_0 ) ] \nonumber
  &= \int_{ \mathbb{S}^{d-1} \times (\RR^d \times \{ 0 \}) } f( m, s ) \, P^\ast(\ds, \dm) \nonumber \\
   &= \int_{ \mathbb{S}^{d-1} \times (\RR^d \times \{ 0 \}) } f( s / \norm{m}, m / \norm{m} ) \, \norm{m}^\alpha \, P(\ds, \dm) \nonumber \\
   &= \EV \left[ f \left( \frac{M_0}{\norm{M_1}}, \frac{M_1}{\norm{M_1}} \right) \, \norm{M_1}^\alpha \, \1_{ \{ M_1 \ne 0 \} } \right],
 \label{eq:bftc:adjoint}
 \end{align}
in the sense that if one expectation exists, then so does the other, the two expectations being equal. This corresponds to equation \eqref{timechange:onestep} which originally motivated the definition of an adjoint distribution. The above formula is the special case $s = 1$ and $t = 0$ of the following result.

\begin{prop}
\label{P:BFTC}
Let $\{M_t:t \in \ZZ\}$ be a $\BFTC(\alpha)$. For all integer $s, t \ge 0$ and for all measurable functions $f : (\RR^d)^{s+1+t} \to \RR$ vanishing on $\{ 0 \} \times (\RR^d)^{s+t}$, the $s+1$ numbers
\begin{equation}
\label{E:BFTC}
  \EV \biggl[ f \biggl( \frac{M_{-s+i}}{\norm{M_i}}, \ldots, \frac{M_{t+i}}{\norm{M_i}} \biggr) \, \norm{M_i}^\alpha \, \1_{ \{ M_i \ne 0 \} } \biggr], \qquad
  i = 0, \ldots, s,
\end{equation}
are all the same, in the sense that if one integral exists, then they all exist and they are equal.
\end{prop} 
\begin{proof}
For $s = 0$ there is nothing to prove, so assume that $s \ge 1$. By definition of the integral, it is sufficient to consider the case where $f$ is nonnegative, in which case the expectations in \eqref{E:BFTC} are always well-defined, possibly equal to infinity.

\emph{Reduction to the case $i \in \{ 0, 1 \}$.}
Suppose first that we can show that the numbers corresponding to $i = 0$ and $i = 1$ in \eqref{E:BFTC} are equal, that is (note that $\|M_0\|=1$),
\begin{equation}
\label{E:BFTC:i=1}
  \EV [ f ( M_{-s}, \ldots, M_t ) ] 
  = \EV \biggl[ f \biggl( \frac{M_{-s+1}}{\norm{M_1}}, \ldots, \frac{M_{t+1}}{\norm{M_1}} \biggr) \, \norm{M_1}^\alpha \, \1_{ \{ M_1 \ne 0 \} } \biggr].
\end{equation}
Take arbitrary $i = 0, \ldots, s-1$. Note that
\[
  \EV \biggl[ f \biggl( \frac{M_{-s+i}}{\norm{M_i}}, \ldots, \frac{M_{t+i}}{\norm{M_i}} \biggr) \, \norm{M_i}^\alpha \, \1_{ \{ M_i \ne 0 \} } \biggr]
  = \EV [ g( M_{-s+i}, \ldots, M_{t+i} )]
\]
for a measurable function $g : (\RR^d)^{s+1+t} \to \RR$ with
that vanishes as soon as its first $d$-tuple of arguments is zero. By \eqref{E:BFTC:i=1} applied to $\tilde{s}=s-i$ and $\tilde{t}=t+i$, we find
\[
  \EV [ g( M_{-s+i}, \ldots, M_{t+i} )]
  = \EV \biggl[ g \biggl( \frac{M_{-s+i+1}}{\norm{M_1}}, \ldots, \frac{M_{t+i+1}}{\norm{M_1}} \biggr) \, \norm{M_1}^\alpha \, \1_{ \{ M_1 \ne 0 \} } \biggr].
\]
By definition of $g$, if $M_1 \ne 0$, then
\begin{align*}
  \lefteqn{
  g \biggl( \frac{M_{-s+i+1}}{\norm{M_1}}, \ldots, \frac{M_{t+i+1}}{\norm{M_1}} \biggr)
  } \\
  &= f \biggl( \frac{M_{-s+i+1} / \norm{M_1}}{\norm{(M_{i+1} / \norm{M_1})}}, \ldots, \frac{M_{t+i+1} / \norm{M_1}}{\norm{(M_{i+1} / \norm{M_1})}}  \biggr) \, \norm{(M_{i+1} / \norm{M_1})}^\alpha \, \1_{ \{ M_{i+1} \ne 0 \} } \\
  &= f \biggl( \frac{M_{-s+i+1}}{\norm{M_{i+1}}}, \ldots, \frac{M_{t+i+1}}{\norm{M_{i+1}}} \biggr) \, \frac{\norm{M_{i+1}}^\alpha}{\norm{M_1}^\alpha} \, \1_{ \{ M_{i+1} \ne 0 \} }.
\end{align*}
Combine the previous three displays to see that
\begin{multline*}
  \EV \biggl[ f \biggl( \frac{M_{-s+i}}{\norm{M_i}}, \ldots, \frac{M_{t+i}}{\norm{M_i}} \biggr) \, \norm{M_i}^\alpha \, \1_{ \{ M_i \ne 0 \} } \biggr] \\
  = \EV \biggl[ f \biggl( \frac{M_{-s+i+1}}{\norm{M_{i+1}}}, \ldots, \frac{M_{t+i+1}}{\norm{M_{i+1}}} \biggr) \, \norm{M_{i+1}}^\alpha \, \1_{ \{ M_1 \ne 0, M_{i+1} \ne 0 \} } \biggr].
\end{multline*}
By definition of the forward chain $(M_t)_{t \ge 0}$, we have $M_{i+1} = 0$ as soon as $M_1 = 0$. As a consequence, we may the suppress the event $\{ M_1 \ne 0 \}$ in the indicator function on the right-hand side, and thus
\begin{multline*}
  \EV \biggl[ f \biggl( \frac{M_{-s+i}}{\norm{M_i}}, \ldots, \frac{M_{t+i}}{\norm{M_i}} \biggr) \, \norm{M_i}^\alpha \, \1_{ \{ M_i \ne 0 \} } \biggr] \\
  = \EV \biggl[ f \biggl( \frac{M_{-s+i+1}}{\norm{M_{i+1}}}, \ldots, \frac{M_{t+i+1}}{\norm{M_{i+1}}} \biggr) \, \norm{M_{i+1}}^\alpha \, \1_{ \{ M_{i+1} \ne 0 \} } \biggr].
\end{multline*}
We conclude that in order to show \eqref{E:BFTC}, it is enough to show \eqref{E:BFTC:i=1}. We will show \eqref{E:BFTC:i=1} by induction on $s \ge 1$.

\emph{Proof of \eqref{E:BFTC:i=1} if $s = 1$.}
We have to show that
\begin{equation}
\label{E:BFTC:i=1,s=1}
  \EV [ f ( M_{-1}, \ldots, M_t ) ] 
  = \EV \biggl[ f \biggl( \frac{M_0}{\norm{M_1}}, \ldots, \frac{M_{t+1}}{\norm{M_1}} \biggr) \, \norm{M_1}^\alpha \, \1_{ \{ M_1 \ne 0 \} } \biggr].
\end{equation}
We will proceed by induction on $t \ge 0$. 

The case $t = 0$ is nothing more than the adjoint relation between the laws of $(M_0, M_1)$ and $(M_0, M_{-1})$, see \eqref{eq:bftc:adjoint}. 

Let $t \ge 1$ and let \eqref{E:BFTC:i=1,s=1} be fulfilled for $t-1$.
By the Markov property,
\[
  \EV [ f ( M_{-1}, \ldots, M_t ) ]
  = \EV [ g ( M_{-1}, \ldots, M_{t-1} ) ]
\]
with
\[
  g( m_{-1}, \ldots, m_{t-1} ) = \EV \{ f ( m_{-1}, \ldots, m_{t-1}, M_t ) \mid M_{t-1} = m_{t-1} \}
\]
As $g(0, m_0, \ldots, m_{t-1}) = 0$, we can apply the induction hypothesis, yielding
\[
  \EV [ g ( M_{-1}, \ldots, M_{t-1} ) ]
  = \EV [ g ( M_0 / \norm{M_1}, \ldots, M_t / \norm{M_1} ) \, \norm{M_1}^\alpha \, \1_{ \{ M_1 \ne 0 \} } ].
\]
The defining property of a $\BFTC$ implies that for every $c > 0$, for every integer $r \ge 1$ and for every nonnegative, measurable function $h$ on $\RR^d$,
\begin{equation}
\label{eq:bftc:scaling:forward}
 \EV [ h( c M_r ) \mid M_{r-1} = m/c ]
  = \begin{cases} h(0) & \text{if $m = 0$,} \\ \EV [ h( \norm{m} M_1 ) \mid M_0 = m / \norm{m} ] & \text{if $m \ne 0$,} \end{cases}
\end{equation}
the right-hand side not depending on the scaling constant $c$ nor on the time index $r$. It follows that if $m_1 \ne 0$,
\begin{align*}
  \lefteqn{
  g ( m_0 / \norm{m_1}, \ldots, m_t / \norm{m_1} )
  } \\
  &= \EV [ f ( m_0 / \norm{m_1}, \ldots, m_t / \norm{m_1}, M_t ) \mid M_{t-1} = m_t / \norm{m_1} ] \\
  &= \EV [ f ( m_0 / \norm{m_1}, \ldots, m_t / \norm{m_1}, M_{t+1} / \norm{m_1} ) \mid M_t = m_t ].
\end{align*}
We find that, on the event $\{M_1 \ne 0\}$, by the Markov property,
\begin{multline*}
  g ( M_0 / \norm{M_1}, \ldots, M_t / \norm{M_1} ) \\
  = \EV [ f ( M_0 / \norm{M_1}, \ldots, M_t / \norm{M_1}, M_{t+1} / \norm{M_1} ) \mid M_0, \ldots, M_t ].
\end{multline*}
We can conclude that
\begin{align*}
  \lefteqn{
  \EV [ f ( M_{-1}, \ldots, M_t ) ]
  } \\
  &= \EV [ g ( M_{-1}, \ldots, M_{t-1} ) ] \\
  &= \EV [ g ( M_0 / \norm{M_1}, \ldots, M_t / \norm{M_1} ) \, \norm{M_1}^\alpha \, \1_{ \{ M_1 \ne 0 \} } ] \\
  &= \EV [ f ( M_0 / \norm{M_1}, \ldots, M_t / \norm{M_1}, M_{t+1} / \norm{M_1} ) \, \norm{M_1}^\alpha \, \1_{ \{ M_1 \ne 0 \} } ],
\end{align*}
as required.

\emph{Proof of \eqref{E:BFTC:i=1} for general $s \ge 1$.} 
The case $s = 1$ was treated above. So let $s \ge 2$. By the Markov property, we have
\[
  \EV [ f ( M_{-s}, \ldots, M_t ) ] = \EV [ g ( M_{-s+1}, \ldots, M_t ) ]
\]
with $g : (\RR^d)^{s+t} \to \RR$ a nonnegative, measurable function defined by
\[
  g ( m_{-s+1}, \ldots, m_t )
  = \EV \{ f ( M_{-s}, m_{-s+1}, \ldots, m_t ) \mid M_{-s+1} = m_{-s+1} \}.
\]
Conditionally on $M_{-s+1} = 0$, we have $M_{-s} = 0$, and thus $f ( M_{-s}, \ldots ) = 0$ too. It follows that $g ( 0, m_{-s+2}, \ldots, m_t ) = 0$. By the induction hypothesis, we therefore have
\begin{equation*}
  \EV [ g ( M_{-s+1}, \ldots, M_t ) ]
  = \EV [ g ( M_{-s+2} / \norm{M_1}, \ldots, M_{t+1} / \norm{M_1} ) \, \norm{M_1}^\alpha \, \1_{ \{ M_1 \ne 0 \} } ].
\end{equation*}
As for the forward chain in \eqref{eq:bftc:scaling:forward}, we have for every nonnegative, measurable function $h$ on $\RR^d$ and every $c>0$,
\begin{equation}
\label{eq:bftc:scaling:backward}
 \EV [ h( c M_{-r} ) \mid M_{-r+1} = m/c ]
  = \begin{cases} h(0) & \text{if $m = 0$,} \\ \EV [ h( \norm{m} M_{-1} ) \mid M_0 = m / \norm{m} ] & \text{if $m \ne 0$,} \end{cases}
\end{equation}
the right-hand side not depending on the scaling constant $c > 0$ nor on the time index $r = 1, 2, \ldots$. It follows that for $m_1 \ne 0$, we have
\begin{align*}
  \lefteqn{
  g ( m_{-s+2} / \norm{m_1}, \ldots, m_{t+1} / \norm{m_1} )
  } \\
  &= \EV [ f ( M_{-s}, m_{-s+2} / \norm{m_1}, \ldots, m_{t+1} / \norm{m_1}) \mid M_{-s+1} = m_{-s+2} / \norm{m_1} ] \\
  &= \EV [ f ( M_{-s+1} / \norm{m_1}, m_{-s+2} / \norm{m_1}, \ldots, m_{t+1} / \norm{m_1}) \mid M_{-s+2} = m_{-s+2} ].
\end{align*}
Invoking the Markov property again, we conclude that
\begin{align*}
    \EV [ f ( M_{-s}, \ldots, M_t ) ]
  &= \EV [ g ( M_{-s+1}, \ldots, M_t ) ] \\
  &= \EV [ g ( M_{-s+2} / \norm{M_1}, \ldots, M_{t+1} / \norm{M_1} ) \, \norm{M_1}^\alpha \, \1_{ \{ M_1 \ne 0 \} } ] \\
  &= \EV [ f ( M_{-s+1} / \norm{M_1}, \ldots, M_{t+1} / \norm{M_1} ) \, \norm{M_1}^\alpha \, \1_{ \{ M_1 \ne 0 \} } ],
\end{align*}
as required. This concludes the proof of Proposition \ref{P:BFTC}.
\end{proof}

The following proposition connects BFTCs and spectral processes. 

\begin{prop}
\label{P:BFTCisspectral}
Let $\{Y_t:t \in \mathbb{Z}\}$ be an $\mathbb{R}^d$-valued process and let $\{M_t:t \in \mathbb{Z}\}$ be an $\mathbb{R}^d$-valued BFTC($\alpha)$. If 
\begin{equation}\label{E:forwardthesame}\mathcal{L}(Y_0, \dots, Y_t)=\mathcal{L}(M_0, \dots, M_t)
\end{equation}
for all $t \geq 0$ and if
\begin{equation}
\label{E:timechange:BFTCspectral}
 \EV\left[f(Y_{-s}, \ldots, Y_{t}) \right]= \EV\left[f\left(\frac{Y_0}{\|Y_s\|}, \ldots, \frac{Y_{s+t}}{\|Y_s\|} \right) \|Y_s\|^\alpha\1_{\{Y_s \neq 0\}} \right]
\end{equation}
for all $s,t \geq 0$ and for all bounded and measurable $f:(\mathbb{R}^d)^{s+t+1} \to \mathbb{R}$ satisfying $f(y_{-s}, \ldots, y_t)=0$ whenever $y_{-s}=0$, then
\begin{equation}\label{E:allthesame} \mathcal{L}(Y_{-s}, \dots, Y_t)=\mathcal{L}(M_{-s}, \dots, M_t)
\end{equation}
for all $s,t \geq 0$. 
\end{prop}
\begin{proof}
The proof relies on the fact that both the process $\{Y_t: t \in \mathbb{Z}\}$ which satisfies \eqref{E:timechange:BFTCspectral} and the BFTC($\alpha$) are uniquely determined by their forward process. Our proof is by induction on $s$. For $s=0$, equation \eqref{E:allthesame} is equal to the assumption \eqref{E:forwardthesame} for all $t \geq 0$. For the induction step, assume that \eqref{E:allthesame} holds for a fixed value of $\tilde{s}=s-1\geq 0$ and all $t \geq 0$. Let $f:(\mathbb{R}^d)^{s+t+1} \to \mathbb{R}$ be a bounded continuous function. Write 
$$f(y_{-s}, \dots, y_t)=f_1(y_{-s}, \dots, y_t)+f_2(y_{-s}, \dots,y_t)$$
with
$$f_1(y_{-s}, \dots, y_t)=f(0,y_{-s+1}, \dots, y_t),$$
$$f_2(y_{-s}, \dots, y_t)=f(y_{-s}, y_{-s+1}, \dots, y_t)-f(0, y_{-s+1}, \dots, y_t).$$
and note that $f_2(0, y_{-s+1}, \dots, y_t)=0$, while the value of $f_1$ does not depend on the first coordinate of the argument.
Then
\begin{eqnarray*}
&& \EV[f(Y_{-s}, \dots, Y_t)]\\
&=& \EV[f_1(Y_{-s}, \dots, Y_t)]+\EV[f_2(Y_{-s}, \dots, Y_t)] \\
&=& \EV[f_1(Y_{-s}, \dots, Y_{t})]+ \EV\left[f_2\left(\frac{Y_0}{\|Y_s\|}, \ldots, \frac{Y_{s+t}}{\|Y_s\|}\right) \|Y_s\|^\alpha\1_{\{Y_s \neq 0\}} \right] \\
&=& \EV[f_1(M_{-s}, \dots, M_{t})]+ \EV\left[f_2\left(\frac{M_0}{\|M_s\|}, \ldots, \frac{M_{s+t}}{\|M_s\|}\right) \|M_s\|^\alpha\1_{\{M_s \neq 0\}} \right] ,
\end{eqnarray*}
where both the induction hypothesis and equations \eqref{E:timechange:BFTCspectral} and \eqref{E:allthesame} have been used. Since $\{M_t: t \in \mathbb{Z}\}$ is a BFTC($\alpha$), we may apply Proposition \ref{P:BFTC} for $i=s$ and $i=0$ (note that $\|M_0\|=1$), so that the above expression is equal to
\begin{equation*} \EV[f_1(M_{-s}, \dots, M_{t})]+ \EV\left[f_2\left(M_{-s}, \ldots, M_{t}\right)\right] =\EV[f(M_{-s}, \dots, M_{t})], 
\end{equation*}
which finishes the induction step and the proof.
\end{proof}
\begin{rem}
Proposition \ref{P:BFTCisspectral} can be read in the following way: Every spectral process $\{M_t:t \in \mathbb{Z}\}$ with a forward process (meaning: $\{M_t:t \in \mathbb{N}_0\}$) which has a BFTC$(\alpha)$ structure, automatically has a BFTC($\alpha$)-backward-distribution as well. This means that a Markovian structure in the forward spectral process (which may also arise in settings where the underlying process is non-Markovian) is enough to secure a Markovian structure of the backward spectral process as well.    
\end{rem}
\begin{cor}
\label{Cor:spectralisBFTC}
Let $\{ X_t : t \in \ZZ\}$ be a stationary Markov chain with distribution determined by \eqref{E:MC:1}, \eqref{E:MC:2} and \eqref{E:MC:3}. Then the corresponding spectral process $\{M_t:t \in \mathbb{Z}\}$ is a BFTC$(\alpha)$. 

We call $\{M_{-t}: t \in \mathbb{N}_0\}$ the \emph{backward tail chain of} $\{X_t:t \in \mathbb{Z}\}$ and $\{M_t: t \in \mathbb{Z}\}$ the \emph{tail chain of} $\{X_t:t \in \mathbb{Z}\}$.
\end{cor}
\begin{proof}
The existence of a corresponding spectral process follows from Proposition \ref{P:BS09:2.1}. Furthermore, it follows from Theorem \ref{T:forward} that the forward process $\{M_t:t \in \mathbb{N}_0\}$ is equal in law to the forward process of a BFTC$(\alpha)$. By Proposition \ref{P:BFTCisspectral} the statement follows. 
\end{proof}

\begin{rem}
Since the forward and backward tail chain of a process $\{X_t:t\in \mathbb{Z}\}$ are uniquely determined by the laws of $(M_0, M_1)$ and $(M_0, M_{-1})$, respectively, it follows that the backward tail chain is equal in distribution to the forward tail chain if and only if the law of $(M_0, M_1)$ is self-adjoint (cf.\ Remark \ref{selfadjoint}). This is for example the case if the process $\{X_t: t \in \mathbb{Z}\}$ fulfills the assumptions of Corollary \ref{Cor:spectralisBFTC} and is in addition a time reversible Markov chain. 

More generally, since the existence of a forward tail process ensures joint regular variation of $(X_0, X_1)$ (cf.\ Corollary 3.2 in \cite{BS09}), the resulting limiting spectral measure of the $2d$-dimensional vector $(X_{0,1}, \ldots, X_{0,d}, X_{1,1}, \ldots, X_{1,d})$ and the law of $(M_0,M_1)$ uniquely determine each other. Therefore, the backward tail chain is equal in distribution to the forward tail chain if and only if the spectral measure of $(X_{0,1}, \ldots, X_{0,d}, X_{1,1}, \ldots, X_{1,d})$ is equal to the spectral measure of $(X_{1,1}, \ldots, X_{1,d}, X_{0,1}, \ldots, X_{0,d})$. For $d=1$ this simply means that the spectral measure of $(X_0,X_1)$ is symmetric.    
\end{rem}

In the univariate case, BFTCs have an additional structure which generalizes a multiplicative  random walk in that the distribution of the increment depends on the sign of the process in its current state \citep{S07}. The random walk structure of the forward tail chain was first observed in \cite{Smith92} for one-sided extremes and extended to allow for both positive and negative extremes in \cite{BC03}.

\section{Examples for BFTCs}
\label{S:examples}

We conclude the paper with some examples of BFTCs for multivariate Markov processes. For univariate examples, see \citet[][section~7]{S07}.

\begin{ex}
\label{Ex:KestenRDE}
Let $(A_t, B_t), t \in \mathbb{Z},$ be i.i.d.\ with $A_t \in \mathbb{R}^{d \times d}$ and $B_t \in \mathbb{R}^d$. The stationary distribution and asymptotic behavior of the corresponding random difference equation 
\begin{equation}\label{RDE} X_{t}=A_t X_{t-1} + B_t, \;\;\; t \in \mathbb{Z},\end{equation}
have been studied initially in the seminal work by \cite{Kesten73}. Let us assume that the distribution of $(A_t, B_t)$ satisfies the technical, but mild assumptions of Theorems~A and~B or Theorem 6 in \cite{Kesten73} (where the first two theorems deal with the nonnegative case, i.e.\ all components of $A_t, t \in \mathbb{Z},$ are nonnegative almost surely, and the last one treats the general case). Together with results in \cite{BL09} this implies that the stationary distribution of $X_t$ for \eqref{RDE} is multivariate regularly varying in the nonnegative case. In the general case, multivariate regular variation follows if $\kappa_1>0$ in \cite{Kesten73}, Equation~(4.8), is not an integer, cf.\ \cite{BDM02a}. Let $\Upsilon$ denote the spectral measure and $\alpha>0$ the index of regular variation of the stationary distribution of $X_t$. It can be shown that
\begin{equation}\label{RDEequation} \EV\left[f\left(\frac{AC}{\| AC\|} \right)\|AC\|^\alpha \right]=\EV[f(C)] \end{equation}
for all bounded, continuous funtions $f$ on $\mathbb{S}^{d-1}$, where $C \in \mathbb{S}^{d-1}$ has distribution $\Upsilon$ and $A \in \mathbb{R}^{d \times d}$ is independent of $C$ with $\mathcal{L}(A)=\mathcal{L}(A_1)$, cf.\ \cite{BS09}. 

Due to the linear structure of \eqref{RDE}, Theorem~\ref{T:forward} applies with $P(Y>y)=y^{-\alpha}, y>1, \mathcal{L}(M_0)=\Upsilon$ and $\phi(M_{j-1}, \epsilon_j)=\epsilon_j M_{j-1}$ where the $\epsilon_j \in \mathbb{R}^{d \times d}, j=1,2,\ldots,$ are i.i.d.\ with $\mathcal{L}(\epsilon_j)=\mathcal{L}(A_1)$. In order to find the distribution of the backward tail chain note that Remark~\ref{partsimple} applies to this example by equation~\eqref{RDEequation}. So the law $P^\ast$ of $(M_{0},M_{-1})$ is given by
\begin{equation}\label{simpleRDEadjoint}P^\ast(E)=\EV\left[ \1_\EV\left(\frac{AC}{\|AC\|},\frac{C}{\|AC\|} \right)\|AC\|^\alpha \right]\end{equation}
for all Borel sets $E \subset \mathbb{S}^{d-1}\times \mathbb{R}^d$.

Additional assumptions about $\mathcal{L}(A)$ allow us to simplify this characterization: Let us assume that $A$ has a multiplicative form like in Example~\ref{Ex:mult:1}, i.e.\ $A=RQ$ for a positive random variable $R$ with $\EV[R^\alpha]=1$ and $Q$ is an orthogonal matrix independent of $R$. We may additionally assume that $R$ has a density on $\mathbb{R}_+$ and that the support of the law of $Q$ is equal to the orthogonal group in dimension $d$. In this case, the spectral measure $\Upsilon$ is the uniform distribution on $\mathbb{S}^{d-1}$ (cf.\ \cite{Bura09}, p.\ 390), $\alpha>0$ is the index of regular variation and
 $$\EV\left[f\left(\frac{AC}{\| AC\|} \right)\|AC\|^\alpha \right]=\EV\left[f\left(QC \right)R^\alpha \right]=\EV\left[f\left(QC \right)\right]\EV[R^\alpha]=\EV[f(C)]$$ 
holds for all bounded, continuous functions $f$ on $\mathbb{S}^{d-1}$ with $C \sim \mbox{Unif}(\mathbb{S}^{d-1})$. Since $\mathcal{L}(C)=\mathcal{L}(QC)$, all assumptions of Example~\ref{Ex:mult:1} are met and the adjoint measure $P^\ast$ is determined by \eqref{E:Ex:mult:1} and equal to the law of $(C^\ast ,R^\ast Q^\ast C^\ast)$ with $R^\ast, Q^\ast, C^\ast$ independent, $\mathcal{L}(C^\ast)=\mbox{Unif}(\mathbb{S}^{d-1})$, $\mathcal{L}(Q^\ast)=\mathcal{L}(Q')$ and $R^\ast$ has density $f_{R^\ast}(y)=f_R(y^{-1})y^{-(2+\alpha)},$ $y>0$, where $f_R$ denotes the density of $R$. Thus, both the forward and the backward tail chain have a simple multiplicative structure:
\[ 
  M_t=M_0A_1 \cdot \ldots \cdot A_t, \qquad 
  M_{-t}=M_0A_{-1}\cdot \ldots \cdot A_{-t}, 
  \qquad t\geq 1, 
\]
with $A_1, A_2, \ldots$ as above and $A_{-1}, A_{-2}, \ldots$ i.i.d. with the same distribution as $R^\ast Q^\ast $, all independent of each other and of $M_0\sim\mbox{Unif}(\mathbb{S}^{d-1})$.
\end{ex}

\begin{ex}
\label{heavy-tailedRDE}
While the preceding example dealt with random difference equations where the random increment $B_t$ has a relatively light tail [\cite{Kesten73} assumes that $E(\|B_1\|^\alpha)<\infty$], the following example deals with AR(1) processes where the innovations themselve are regularly varying. Let 
\begin{equation}
\label{E:StochDiff:2}
	X_t = A X_{t-1} + B_t, \qquad t \in \mathbb{Z},
\end{equation}
where $A$ is a deterministic $\mathbb{R}^{d \times d}$-matrix and $B_t \in \mathbb{R}^d$, $t \in \mathbb{Z}$, are i.i.d.\ and multivariate regularly varying with index $\alpha>0$ and spectral measure $\lambda$ on $\mathbb{S}^{d-1}$. For extensions to random but light-tailed random matrices $A_t$, see for instance \citet{HuSa08}.

If $\sup_{x \in \mathbb{S}^{d-1}}\|A^m x\|<1$ for some positive integer $m$, then \eqref{E:StochDiff:2} has the stationary solution
\[ 
  X_t=\sum_{n=0}^\infty A^n B_{t-n}, \qquad t \in \mathbb{Z}.
\]
It has been shown in \cite{MeSe10} that in this case the stationary distribution of $X_t$ is multivariate regularly varying as well, with the same index $\alpha$ and spectral measure $\Upsilon=\sum_{n=0}^\infty p_n \lambda_n$, where
$$ p_n:= \frac{c_n}{\sum_{k=0}^\infty c_k} \;\;\; \text{with} \;\; c_n:=\int_{\mathbb{S}^{d-1}}\|A^n \theta\|\, \lambda (\mathrm{d} \theta), \;\;\; n \in \mathbb{N}_0,$$
and where $\lambda_n$ is the spectral measure of $A^n B$, provided $c_n > 0$, i.e.
\[
  \lambda_n(f)
  :=
  \frac{1}{c_n}
  \int_{\mathbb{S}^{d-1}}
    f \left( \frac{A^n s}{\|A^n s\|} \right) \, \|A^n s\|^\alpha \, 
  \lambda (\mathrm{d} s),
  \;\;\; n \in \mathbb{N}_0, \;\; \mbox{if }c_n>0,
\]
for all bounded, continuous functions $f$ on $\mathbb{S}^{d-1}$ \citep[][Example~9.3]{MeSe10}. The spectral process $\{ M_t:t \in \mathbb{Z}\}$ in Proposition~\ref{P:BS09:2.1} is of the form
\begin{equation}
\label{E:AR(1)tailproc}
  M_{-N+t}=\begin{cases}
            A^t \Theta, \;\;\; & t=0, 1, 2, \ldots ,\\
            0, & t=-1,-2, \ldots
           \end{cases}
\end{equation}
for a random integer $N$ with $\Pr(N=n)=p_n$, $n \in \mathbb{N}_0$, and a random vector $\Theta$ with distribution
\[ 
  \Pr(\Theta \in E \mid N=n)
  = \frac{1}{c_n}
  \int_{\mathbb{S}^{d-1}}
    \1_{E} (s / \|A^n s\|) \, \|A^n s\|^\alpha \,
  \lambda (\mathrm{d} s)
\] 
for $n \in \mathbb{N}_0$ and Borel sets $E \in \mathbb{R}^d$. Here, the forward tail chain has a deterministic multiplicative structure with $M_0 \sim \Upsilon$ and $M_n=AM_{n-1}$ for $n\geq 1$. The backward process is Markovian as well, by Corollary~\ref{Cor:spectralisBFTC}. This is also clear if one looks at \eqref{E:AR(1)tailproc} and notices that $M_{-(n+h)}=0$ if $M_{-n}=0$ for all $h \geq 1, n \geq 1$. Furthermore, if $M_{-n} \neq 0$ then $(M_{-n+1}, \ldots, M_0)=(AM_{-n}, \ldots, A^n M_{-n})$ contains no more information about $M_{-(n+1)}$ than $M_{-n}$ does.

The distribution of $(M_0, M_{-1})$ is adjoint to the one of $(M_0, M_1) = (M_0, AM_0)$. By \eqref{eq:adjoint:E} and since $M_0 \sim \Upsilon$, we find, for every Borel set $E \subset \mathbb{S}^{d-1} \times (\mathbb{R}^d \setminus \{ 0 \})$,
\begin{align*}
  \Pr\bigl( (M_0, M_{-1}) \in E \bigr)
  &= \EV 
  \biggl[ 
    \1_E \biggl( \frac{M_1}{\|M_1\|}, \frac{M_0}{\|M_1\|} \biggr) \, \|M_1\|^\alpha 
  \biggr] \\
  &= \frac{1}{\sum_{k = 0}^\infty c_k}
  \sum_{n \ge 0} 
  \int_{\mathbb{S}^{d-1}}
    \1_E \biggl( \frac{A^{n+1}s}{\|A^{n+1}s\|}, \, \frac{A^n s}{\|A^{n+1} s\|} \biggr) \,
    \| A^{n+1} s \|^\alpha \,
  \lambda( \mathrm{d}s ).
\end{align*}
Choosing $E = S \times (\mathbb{R}^d \setminus \{ 0 \})$ for a Borel set $S \subset \mathbb{S}^{d-1}$ yields, upon taking complements with respect to $\{ M_0 \in S \}$ and noting that $\| s \| = 1$ for $s \in \mathbb{S}^{d-1}$,
\begin{equation}
\label{E:M-1equal0}
  \Pr( M_0 \in S, \, M_{-1} = 0 )
  = \frac{1}{\sum_{k = 0}^\infty c_k} \, \lambda(S).
\end{equation}
In particular, $\Pr(M_{-1} = 0) = p_0 = \Pr(N = 0)$. The backward tail chain now follows from Definition~\ref{D:BFTC}(iii) together with the distribution of $(M_0, M_{-1})$.

In the special case that $A$ is invertible, we find from \eqref{E:AR(1)tailproc} that $M_{-(t+1)}$ is equal to either $A^{-1} M_{-t}$ or to $0$ with conditional probabilities depending on $M_{-t}/\|M_{-t}\|$: if $M_{-t} = 0$, then $M_{-(t+1)} = 0$ too, while if $M_{-t} = x \ne 0$, then
\[
  M_{-(t+1)} =
  \begin{cases}
    A^{-1} x & \text{with probability $1 - \Pr(M_{-1} = 0 \mid M_0 = x / \norm{x})$,} \\
    0 & \text{with probability $\Pr(M_{-1} = 0 \mid M_0 = x / \norm{x})$.}
  \end{cases}
\]
To derive a concrete form of the backward Markov kernel, let us assume that $\lambda$ has a Lebesgue density $f_\lambda$ on $\mathbb{S}^{d-1}$. Then all measures $\lambda_n$ and thus $\Upsilon$ have Lebesgue densities as well and \eqref{E:M-1equal0} gives us
\[
  \Pr(M_{-1} = 0 \mid M_0 = s)
  = \frac{1}{\sum_{k = 0}^\infty c_k} \frac{f_\lambda(s)}{f_\Upsilon(s)}
\]
for all $s \in \mathbb{S}^{d-1}$ such that $f_\Upsilon(s) > 0$.

\end{ex}

\section*{Acknowledgement} 
The authors thank Richard Davis and Holger Drees for helpful discussions. Furthermore, they wish to thank the organisers, especially Paul Doukhan, of the workshop ``Extremes and risk management`` which took place during September 2012 at the university of Cergy-Pontoise. Anja Jan\ss en was supported by DFG (DFG project JA 2160/1-1). Johan Segers was supported by contract ``Projet d'Actions de Recherche Concert\'ees'' No.\ 12/17-045 of the ``Communaut\'e fran\c{c}aise de Belgique'' and by IAP research network grant No.\ P7/06 of the Belgian government (Belgian Science Policy).

The authors would like to thank the anonymous referee for helpful comments and suggestions. 


\end{document}